\documentclass[12pt]{article}

\input epsf
\usepackage{amsmath}
\usepackage{amssymb}
\usepackage{amsthm}

\usepackage{graphicx}
\usepackage[arrow, matrix, curve]{xy}

\usepackage[margin=3cm]{geometry}
\usepackage{array, xcolor}
	\definecolor{lightgray}{gray}{0.8}
	\newcolumntype{L}{>{\raggedleft}p{0.14\textwidth}}
	\newcolumntype{R}{p{0.8\textwidth}}
	\usepackage[ansinew]{inputenc}

\newtheorem{theorem}{Theorem}

\newtheorem{lemma}[theorem]{Lemma}

\newtheorem{proposition}[theorem]{Proposition}
\theoremstyle{remark}
\theoremstyle{definition}\newtheorem{definition}[theorem]{Definition}

\newcommand{\R}{\mathbb R}

\newcommand{\id}{\operatorname{id}}

\newcommand{\sd}{\kern-.1em\cdot\kern-.1em}

\def\3{\ss}

\def\Hess{\operatorname{Hess}}

\def\aplim{\operatorname{aplim}}
\def\tr{\operatorname{tr}}

\def\D{\partial}

\def\2pithird{\frac{2\pi}{3}}

\def\D{\partial}

\title{Conformal deformations of CAT(0) spaces}
\author{Alexander Lytchak, Stephan Stadler}

\begin{document}

\maketitle


\begin{abstract}
We show that the class of CAT(0) spaces is closed under  suitable conformal changes.
 In particular, any CAT(0) space admits a large variety of non-trivial deformations.
\end{abstract}

\section{Introduction}

\subsection{Main result}

Let  $M$ be a smooth manifold with Riemannian metric $g$. It is well known that a conformal change
$\bar g=e^{2f}\cdot g$ affects the sectional curvature $K$ in the following way:
\[e^{2f}\bar K=K-\left(\Hess f(X,X)+\Hess f(Y,Y)+|\nabla f|_g^2-
g(\nabla f,X)^2-g(\nabla f,Y)^2\right),\]
where $X,Y$ is  an orthonormal basis of the respective tangent plane.
So if $f$ is a convex function
then $\bar K\leq e^{-2f}\cdot K$. Since complete, simply connected Riemannian manifolds of non-positive sectional curvature, so-called Hadamard manifolds,
support many convex functions, the above formula shows that these spaces admit non-trivial conformal deformations.
In this paper we generalize this result to synthetic analogs of Hadamard manifolds, namely CAT(0) spaces.
For a length space $X$ and a continuous function $f$ on $X$
we obtain a \emph{conformally equivalent length space} $e^f\sd X$ by stretching the lengths of curves according to the weight $e^f$.

\begin{theorem}\label{thm:main}
Let $X$ be a CAT(0) space and  $f$ be a function on $X$, continuous, convex and bounded from below.
Then  the conformally equivalent space  $e^f\sd X$ is CAT(0).
\end{theorem}

While the class of CAT(0) spaces is closed under gluings along convex subsets, warped products with convex warping functions and ultralimits,
Theorem \ref{thm:main} is the first result which allows to deform a general CAT(0) space. In contrast, it is still unknown if Alexandrov
spaces with curvatures bounded from below admit any non-trivial deformations, see \cite{Psemi}, Section 9.  We would like to mention the recently achieved control
 of synthetic Ricci curvature bounds under conformal changes by Theo Sturm \cite{Sturm}.
There is a reason why Theorem \ref{thm:main} is difficult
and inaccessible by known techniques. As opposed to known constructions, for instance warped products  \cite{ABcurv},
when changing a metric conformally, information about all geodesics is lost.

To overcome this problem we rely on the theory of minimal surfaces
in metric spaces and the recently obtained structural results from \cite{LWint}, \cite{LWcurv} and \cite{PS}.
Instead of looking at geodesic triangles we use a CAT(0) recognition  statement based on Reshetnyak's majorization theorem. More precisely, we
show that any rectifiable Jordan curve in our conformally changed  space is majorized by some two-dimensional non-positively curved disc.
In order to find this disc, we  solve Plateau's problem for a given Jordan curve in the original space $X$. By the aforementioned structure results,
the minimal disc obtained in this way will intrinsically be non-positively curved.
If we now perform a conformal change, our disc will not be minimal anymore. However, we will prove that it remains non-positively curved after the conformal change. As a main tool we rely on
the control of  curvature of flat surfaces under conformal changes  obtained  by Yuri Reshetnyak in the sixties.
In order to apply Reshetnyak's results, we will
 use the fact that the pull-back of a convex function by a harmonic map results in a subharmonic function \cite{F}.

We expect our result to be useful for extremal problems within the realm of non-positive curvature,  similar to
  \cite{Psemi}, Section 9.

It seems to be only a  technical issue to extend our result in the following directions, which we leave to interested readers.

\begin{itemize}
\item Theorem \ref{thm:main}  extends to non-zero curvature bounds $K$ and semi-convex functions $f$, according to the
formula valid in the smooth case.
\item  Theorem \ref{thm:main} should be valid if $f$ is a semi-convex function, such that the gradient flow of $-f$
does not increase areas.  In the smooth case this corresponds to the assumption that the sum of the first two eigenvalues of  the Hessian of
$f$  is non-negative.
\item  A variant of  Theorem \ref{thm:main} should be valid if $f$  is merely lower semi-continuous and not bounded from below.
\end{itemize}

\subsection{Structure of the paper}

In Section 2 we introduce notation and recall basic facts from metric geometry.
 We include an easy CAT(0) recognition  statement  used in the proof of
our main result.

In Section 3 we introduce conformal changes for length spaces, discuss
Reshetnyak's findings on conformal changes of Euclidean domains
and study the change of upper curvature bounds under repeated conformal changes.

In Section 4 we review basics of Sobolev  maps with metric space targets.
We recall the solution of Plateau's problem in metric spaces and the intrinsic structure of
minimal discs.  We show how minimal discs yield majorizations of Jordan curves
in CAT(0) spaces. Finally, we give the proof of our main result.

\subsection{Acknowledgment}
The first and second author were partially supported by DFG grant  SPP 2026.

The authors are grateful to Stefan Wenger for very helpful comments.

\section{Metric geometry}

\subsection{Basics and notation}

We refer the reader to \cite{BBI}, \cite{B} and \cite{AKP} for basics on metric geometry and
CAT(0) spaces. Here we just agree on notation and recall some important facts.
As usual $\R^2$ will denote the Euclidean plane.
We will let $D\subset\R^2$ be the open Euclidean unit disc, $\bar D\subset\R^2$ the closed Euclidean unit disc and
$S^1=\D\bar D\subset\R^2$ the unit circle.

Let $X$ be a metric space. The metric on $X$ will be denoted by $|\cdot,\cdot|_X$ and if there is no risk of
confusion by $|\cdot,\cdot|$. If $Y$ is another metric space and $f:X\to Y$ is a map, then $f$ is {\em Lipschitz continuous} if
there is a constant $C>0$ such that $|f(x),f(y)|_Y\leq C\cdot|x,y|_X$. If the constant $C$ can be chosen to be one, then
$f$ will be called {\em short}.

The length of a rectifiable curve $\gamma$ in $X$ is denoted by $L(\gamma)$.
$X$ is called a {\em length space}
 if the distance between any two points
is equal to the greatest lower bound for lengths of curves connecting the respective points.
A curve $c:[a,b]\to X$ will  be called {\em geodesic} if  it is an isometric embedding.
 The space $X$ itself will be called {\em geodesic} if any two points in $X$
are joined by a geodesic.

A {\em triangle} in $X$ consists of three points and three geodesics connecting them.
The three geodesics that make a triangle are called its {\em sides}.
For every triangle $\triangle$, we can find a \emph{comparison triangle} $\tilde\triangle\subset\R^2$
such that corresponding sides have equal lengths. If $X$ is a geodesic space such that
for each triangle $\triangle$ in  $X$ the distances between points on $\triangle$
do not exceed the distances between the corresponding points on $\tilde\triangle$, then $X$
is called a {\em CAT(0) space}. Examples of CAT(0) spaces include simply connected Riemannian manifolds of
non-positive sectional curvature, metric trees and Euclidean buildings. CAT(0) spaces enjoy the uniqueness of geodesics
between points. Moreover, CAT(0) spaces contain many convex subsets: geodesics, metric balls, horoballs, etc.
A function $f$ on a CAT(0) space is called {\em convex}, if it restricts to a conventional convex function
on each geodesic. Since the distance to a closed convex subset
is a convex function, any CAT(0) space supports many convex functions.

 For us it will be important that CAT(0) spaces can be recognized
without referring to geodesic triangles.
By a \emph{Jordan curve} in a metric space $X$ we denote a subset
homeomorphic to a circle.  We say that a metric space $Y$ majorizes a  rectifiable Jordan curve $\Gamma$ in a metric space $X$
if there exists a short map $P:Y\to X$ which sends a Jordan curve $\Gamma'  \subset Y$ bijectively in an arc length preserving way onto $\Gamma$.

\begin{proposition}  \label{prop:major}

Let $X$ be a complete length metric space.  If any rectifiable Jordan curve $\Gamma$ in $X$
is majorized by some CAT(0) space $Y_{\Gamma}$, then $X$ is CAT(0).

\end{proposition}

\begin{proof}
If $X$ is geodesic the  statement is exactly Lemma 3.3 in \cite{LWcurv}. In particular, this includes the case of a locally compact space $X$.
%

 In order to prove the "if direction" in the  not locally compact case, it remains to prove that any $X$ satisfying the
 majorization assumption must be  geodesic.
 So it is enough to prove the following. If two simple curves $c^+$ and $c^-$ with common endpoints $x$ and $y$ almost realize the distance $|x,y|$,
 then their images $I^+$ and $I^-$ are close to each other. Assume that $p$ is a point on $I^+$ with positive distance from $I^-$.
 Then there exist simple arcs $\gamma^\pm\subset I^\pm$ such that $p\in \gamma^+$ and the union $\Gamma:=\gamma^+\cup\gamma^-$ is a Jordan curve.
 Let $Y_\Gamma$ be a CAT(0) space which majorizes $\Gamma$. Further, let $\Gamma_0\subset Y_\Gamma$ be the  Jordan curve corresponding to $\Gamma$.
 Then $\Gamma_0$ decomposes into a union of two arcs $\gamma_0^\pm$ according to the decomposition of $\Gamma$.
 By assumption, $\gamma_0^+$ and $\gamma_0^-$ almost realize the distance between their endpoints. Since $Y_\Gamma$ is CAT(0),
 they have to stay close to each other. Therefore the same is true for $\gamma^+$ and $\gamma^-$.
  \end{proof}

A  metric space $X$ is called {\em non-positively curved}
if every point in $X$ has a neighborhood which is CAT(0).

\section{Conformal changes}

\subsection{Generalities}

Let $X$ be a length space and  $f:X\to (0,\infty)$ be a  continuous function.
We obtain a new length structure on $X$ as follows. As admissible curves we take Lipschitz continuous paths $\gamma:[a,b]\to X$
and define their $f$-length in $X$     by
\begin{equation} \label{eq:length}
L_f(\gamma) = L^X_f(\gamma)=\int_a^b f(\gamma(t)) \cdot |\dot\gamma(t)| \,dt \,,
\end{equation}
where $|\dot\gamma(t)|$ denotes the velocity of the curve $\gamma$ at  time $t$.

Since $f$ is locally bounded away from $0$, the associated
pseudo metric
\begin{equation} \label{eq:dist}
d_f (x,y)=  \inf _{\gamma} \left\{ L_f(\gamma) \; ; \;  \gamma \;  \text{Lipschitz curve from  } x \; \text{ to } \;  y\right\}
\end{equation}
is indeed a metric. We set  $f\sd X:=(X,d_f)$ and call it the \emph{metric space conformally    equivalent to $X$ with conformal factor $f$}.

Because $f$ is locally bounded away from $0$ and $\infty$, the identity map  $\id_f:X\to f\sd X$ is  a  locally bi-Lipschitz homeomorphism.
 If $f$  is bounded from below by a positive constant then any Cauchy sequence in $f\sd X$ is Cauchy in $X$. Hence, if $X$ is
complete and $f$ bounded from below by a positive constant then $f\sd X$ is complete.

 Due to the continuity of $f$, for any Lipschitz curve $\gamma :[a,b]\to X$, the composition $\id_f\circ \gamma :[a,b]\to f\sd X$
 has  at almost all times $t$ the velocity $f(\gamma (t))\cdot  |\dot\gamma(t)| $ and  length equal to $L_f^X (\gamma )$ with
respect to the metric $d_f$.  In particular, $(X,d_f)$ is a length space.

\subsection{Surfaces}

In the case of flat domains the curvature of conformally changed metrics has been investigated
in detail by Yuri Reshetnyak, see \cite{R} and the references therein.  It turns our that in this case it is even possible
to relax the continuity and positivity assumptions on conformal factors.

Recall that a function $f:U\to [-\infty,\infty)$ on a domain $U\subset \R^2$  is called {\em subharmonic},
if it is upper semi-continuous, contained in $L^1_{loc}$ and satisfies $$f(z)\leq\frac{1}{\pi s^2}\int_{B_s(z)}f(w) \, dw$$
for all balls $B_s(z)\subset U$. On the other hand, a function $f\in L^1_{loc}$ has a  subharmonic representative if and only if
$\Delta f\geq 0$ weakly.

 A function $f:U\to[0,\infty)$ is called {\em log-subharmonic}, if $\log(f)$ is subharmonic.
Each log-subharmonic function is locally bounded.
The set of log-subharmonic functions is closed under products.

For a log-subharmonic function $f$ one can use the same formulas \eqref{eq:length} and \eqref{eq:dist}
to define the conformally changed metric on $U$. Indeed, we have the following result due to Reshetnyak,
see Theorem 7.1.1 in \cite{R}, compare also Theorem 8.1 in \cite{LWcurv}.

\begin{theorem}\label{thm:resh}
 Let $U\subset\R^2$ be a domain and $f$ a log-subharmonic function on $U$.
 Then $f\sd U$ has non-positive curvature and $\id_f:U\to f\sd U$
 is a homeomorphism.
\end{theorem}

We can now state and prove the main result of this section.

\begin{lemma} \label{lem:cat}
 Let $U$ be a domain in $\R^2$ with its intrinsic metric.
Let  $\lambda $ be a log-subharmonic function on $U$ and let $\lambda \sd U$ be
the conformally changed space. Let $Y$ denote the completion of $\lambda \sd U$.
Finally, let $f$ be a positive, continuous function on $Y$ such that the restriction of $f$ to $U$
is log-subharmonic.  Then $f\sd Y$ is CAT(0).
\end{lemma}

\begin{proof}
Any Lipschitz curve $\gamma :[a,b]\to Y$ can be approximated by Lipschitz curves
$\gamma _n:[a,b]\to \lambda \sd U$  such that the lengths of $\gamma _n$ converge to the length of $\gamma$.
Since the conformal factor is continuous, this implies that $f\sd Y$ is the completion of the length space
$f\sd (\lambda \sd U)$.

Therefore, the statement of the theorem follows from \cite{LWcurv}, Proposition 12.1, once we can prove that
$f\sd (\lambda \sd U)$ is non-positively curved.

By assumption,  $(f\sd\lambda)$ is  log-subharmonic in $U$.  We  apply Theorem \ref{thm:resh}  and deduce that
$(f\sd \lambda) \sd U$ is non-positively curved. Hence, we only need to verify that
$(f\sd\lambda)\sd U$ is isometric to $f\sd(\lambda\sd U)$.

 Due to Theorem \ref{thm:resh}
 the identity maps $U\to (f\sd\lambda)\sd U$ and $ U\to\lambda\sd U$ are homeomorphisms.
  Moreover, since log-subharmonic
 functions are locally bounded, both maps
 are locally Lipschitz continuous.
  By the continuity and positivity of $f$, the identity $\lambda \sd U \to  f\sd (\lambda \sd U)$ is a locally bi-Lipschitz homeomorphism.

  Consider the natural map $I:(f\sd\lambda)\sd U\to f\sd(\lambda\sd U)$, the composition of the homeomorphisms above.
  Any Lipschitz curve $\gamma :[a,b]\to U$  in the original disc $U$ has  the same length in $(f\sd\lambda)\sd U$ and in
  $ f\sd(\lambda\sd U)$.    This observation, the definition of the distance in the conformally changed metrics and  the
	continuity of $f$ now imply  that $I$ is an isometry.

  This finishes the proof of the lemma.
\end{proof}

\section{Sobolev discs and energy}
\subsection{Generalities  and harmonic discs}
By now there exists a well established theory of Sobolev maps with
 values in metric spaces.
We refer the reader to \cite{KS}, \cite{J},  \cite{HKST}, \cite{LWplateau} and restrict our revision to the special case needed in this paper.

Let $X$ be a complete metric space. We let $L^2(D,X)$ be the set of measurable and essentially
separably valued maps $u:D\to X$ such that for some and thus every $x\in X$ the function $u_x(z):=|x,u(z)|$
belongs to the classical space $L^2(D)$ of  square-integrable functions on $D$.

\begin{definition}
 A map $u\in L^2(D,X)$ belongs to the Sobolev space $W^{1,2}(D,X)$ if there exists  $h\in L^2(D)$ such that for every $x\in X$
 the composition $u_x$
 is contained in the classical Sobolev space $W^{1,2}(D)$ and its weak gradient satisfies $|\nabla u_x|\leq h$ almost everywhere.
 A map $u\in W^{1,2}(D,X)$ will  occasionally be called a {\em Sobolev disc}.
\end{definition}

Each Sobolev disc $u$ has an associated {\em trace} $\tr(u)\in L^2(S^1,X)$, see
 \cite{KS}, \cite{LWplateau}. If $u$ extends continuously to a map $\hat u$ on $\bar D$, then
$\tr(u)$ is represented by the restriction $\hat u|_{\D \bar D}$.
Every map $u\in W^{1,2}(D,X)$ has approximate metric differentials almost everywhere
in $D$,  see \cite{LWplateau}, Section 4.  More precisely,
for almost every $z\in D$ there exists a unique seminorm on $\R^2$, denoted by $|du_z(\cdot)|$, such that
\[\aplim_{z'\to z}\frac{|u(z'),u(z)|-|du_z(z'-z)|}{|z'-z|}=0,\]
where $\aplim$ denotes the approximate limit, see \cite{EG}.

If $X$ is a CAT(0)  space then for all $u\in W^{1,2} (D,X)$ and  almost all $z\in D$ the approximate metric differential comes from a possibly degenerate scalar product, \cite{KS}.
Thus, $X$ satisfies the property ET in the terminology of  \cite{LWplateau}, Section 11.

There are several natural definitions of energy for Sobolev maps, see \cite{LWplateau}.
We will only use the {\em Korevaar-Schoen energy}, defined by
$$E^2(u):=\int_D\mathcal{I}^2_{avg}(|du_z|) \, dz\,,$$
where for a seminorm $s$ on $\R^2$ we have set $\mathcal{I}^2_{avg}(s):=\frac{1}{\pi}\int_{S^1}s(\theta)^2 d\theta$.

  Korevaar and Schoen solved the following Dirichlet problem in \cite{KS}.

\begin{theorem}\label{thm:dirichlet}
 Let $X$ be a CAT(0) space and $v\in W^{1,2}(D,X)$ be a given Sobolev disc.
 Define
 \[W^{1,2}_v:=\{v'\in W^{1,2}(D,X)|\ \tr(v')=\tr(v)\}.\]
 Then there exists a unique  harmonic  disc $u\in W^{1,2}_v$, i.e.
 \[E^2(u)=\inf_{u'\in W^{1,2}_v}E^2(u').\]
 The map $u$ has a representative which is locally Lipschitz continuous in $D$.
\end{theorem}

We will use the following result, a special case of Theorem 2 (b) in \cite{F}.

\begin{theorem} \label{thm:restrict}
Let $X$ be a CAT(0) space and let $u:D\to X$ be a harmonic map.  Let $f :X\to \mathbb R$ be a continuous,
convex function.  Then the composition $f\circ u$ is subharmonic on $D$.
\end{theorem}

\subsection{Minimal discs}

Let $X$ be a CAT(0) space and $\Gamma$ a Jordan curve of finite length in $X$.
Consider the non-empty set $\Lambda(\Gamma,X)$ of all maps $v\in W^{1,2}(D,X)$
such that $\tr(v)$ is a weakly monotone parametrization of $\Gamma$. The classical Plateau problem
asks for an element of least area in $\Lambda(\Gamma,X)$.
We refer to \cite{LWplateau} for the definition of area, not needed in the sequel, and just state that
for  CAT(0) spaces, such a minimal disc is an element of least energy in   $\Lambda(\Gamma,X)$, \cite{LWplateau}, Theorem 11.4.

The following solution of Plateau's problem is  a consequence of
\cite{GW}, Theorem 1.4; \cite{LWplateau}, Theorem 4.2;   and \cite{LWcurv}, Section 7, see also \cite{Mese}.

\begin{theorem}\label{thm:plateau}
 Let $X$ be a CAT(0) space and $\Gamma\subset X$ a rectifiable Jordan curve.
 Then there exists a map $u\in\Lambda(\Gamma,X)$  of least energy, i.e.
 \[E^2(u)=\inf_{u'\in\Lambda(\Gamma,X)}E^2(u').\]
 Every such map has the following properties.
 \begin{enumerate}
  \item[1.)] $u$ is harmonic.
  \item[2.)] $u$  extends continuously to $\bar D$.
  \item[3.)] $u$ is {\em conformal} in the sense that $|du_z|=\lambda\cdot s_0$ with $\lambda\in L^2(D)$ holds for almost all $z\in D$, where
  $s_0$ is the Euclidean norm. The function $\lambda$ is called the {\em conformal factor} of $u$.
  \item[4.)] The conformal factor $\lambda$ has a log-subharmonic representative.
 \end{enumerate}

\end{theorem}

A map $u$ as above will be called a {\em minimal disc} filling $\Gamma$. We agree that the conformal factor
of a minimal disc will always be chosen to be log-subharmonic.

Each minimal disc comes with a nice intrinsic structure  defined in \cite{LWint}. Its relevant properties are summarized in
the following result,  whose proof is just a list of references to \cite{LWint} and \cite{LWcurv}.

\begin{theorem} \label{thm:intrinsic}
 Let $X$ be a CAT(0) space and $\Gamma\subset X$ a  rectifiable Jordan curve.
 Let $u:\bar D\to X$ be a minimal disc filling $\Gamma$ and let the log-subharmonic function $\lambda$
 denote the conformal factor  of $u$.
 Then the completion $Y$ of $\lambda \sd D$ is a CAT(0) space homeomorphic to
 $\bar D$.
 The map $u:\lambda \sd D\to X$ extends to a majorization $v:Y\to X$  of  $\Gamma$.
\end{theorem}

\begin{proof}
By Theorem \ref{thm:resh}, the space $\lambda \sd D$ is non-positively curved and by \cite{LWcurv}, Proposition 12.1  its completion
$Y$ is a CAT(0) space.

In \cite{LWint} an intrinsic structure on the disc $\bar D$  canonically defined by the map $u$ is investigated.
The properties of the arising metric space $Z$ are  investigated in \cite{LWint} and
 summarized in \cite{LWcurv}, Theorem 6.2.    The space $Z$ is   a geodesic space, homeomorphic to $\bar D$.
There exist canonical maps $P:\bar D\to  Z$ and $\bar u:Z\to X$ such that $\bar u$ is a majorization of $\Gamma$ and such that
$u=\bar u \circ P$.

The space $Z$ is a CAT(0) space and  the complement  $Z_0=Z\setminus \partial Z$ of the boundary circle $\partial Z$ in $Z$
is a length space, with respect to the metric induced from $Z$, see \cite{LWcurv}, Theorems 1.2 and 1.3.
In particular, $Z$ is the completion of $Z_0$.

The   preimage $ D_0=P^{-1} (Z_0)$ is a topological disc in $D$, \cite{LWcurv}, Theorem 6.2.
Moreover, the map $P:\lambda \sd D \to Z$   is short and the map  $P:\lambda \sd D_0 \to Z_0$  is an isometry, \cite{LWcurv}, Section 9.
Since the canonical inclusion $\lambda \sd D_0 \to \lambda \sd D$ is short, the last two properties of the map $P$ imply
that the inclusion $\lambda \sd D_0 \to \lambda \sd D$ is an isometric embedding.

The log-subharmonic function $\lambda$ is positive almost everywhere on
$D$ and the map $u$ sends $\bar D \setminus D_0$ onto $\Gamma$. Therefore, the area of   $D\setminus D_0$ is $0$.
Hence $D_0$ is dense in $D$. Therefore, the completion of $\lambda \sd D_0$ (which is via the map $v$ isometric to $Z$)
and the completion $Y$ of $\lambda \sd D$ coincide.

This finishes the proof of the theorem.
\end{proof}

\subsection{Proof of the main theorem}
We can now  assemble all the pieces to a proof  of our main theorem.
\begin{proof}[Proof of Theorem \ref{thm:main}]
Let there be given a CAT(0) space $X$ and a continuous, convex function $f:X\to \R$,   bounded from below.

The positive, continuous  function $e^f$ is bounded away from $0$, hence the conformally changed space $e^f\sd X$ is a
complete length space.

 Consider an arbitrary rectifiable  Jordan curve $\Gamma \subset   e^f\sd  X$.  Applying Proposition \ref{prop:major} we  only need to majorize $\Gamma$ by
 some CAT(0) space.  In order to find this majorizing space, we identify $\Gamma$ with a Jordan curve $\hat \Gamma =\id ^{-1} _{e^f} (\Gamma)$
 in $X$. Since $e^f$ is bounded on $\Gamma$, the Jordan curve $\hat \Gamma$ is rectifiable in $X$.

 Consider a minimal filling $u:\bar D\to X$ of the Jordan curve $\hat \Gamma$ provided by Theorem
\ref{thm:plateau}. Let the log-subharmonic function $\lambda$ be the conformal factor of $u$.   Denote by $Y$ the   completion of $\lambda \sd D$.  Let $v:Y\to X$ be the extension
of the map $u:\lambda \sd D\to X$ to $Y$, which is   a   majorization of $\hat \Gamma$ due to
Theorem \ref{thm:intrinsic} .


The function $e^f\circ v :Y\to \R$ is continuous. Due to   Theorem \ref{thm:restrict}, the function $e^f\circ u$ is log-subharmonic on $D$.
Hence, by Lemma \ref{lem:cat} the conformally changed space $(e^f\circ v ) \sd Y$ is CAT(0).

Since $v:Y\to X$ provides a majorization of  $\hat \Gamma$, the map
$$\id _{e^f} \circ v:(e^f\circ v)\sd Y\to e^f \sd X$$
provides a majorization of $\Gamma$ by the CAT(0) space $(e^f\circ v ) \sd Y$.

Since $\Gamma$ was arbitrary, Proposition \ref{prop:major} shows that $e^f\sd X$ is CAT(0).
This  finishes the proof of Theorem \ref{thm:main}.
\end{proof}


\begin{thebibliography}{alpha}





\bibitem{ABcurv} Stephanie Alexander, Richard Bishop.
\textit{Curvature bounds for warped products of metric spaces.}
GAFA, 14 (2004) 1143-1181.


\bibitem{AKP} Stephanie Alexander, Vitali Kapovitch and Anton Petrunin.
\textit{Alexandrov geometry.}
http://anton-petrunin.github.io/book/all.pdf, 2017.



\bibitem{B} Werner Ballmann.
\textit{On the geometry of metric spaces.}
Preprint, lecture notes,
http://people.mpim-bonn.mpg.de/hwbllmnn/archiv/sin40827.pdf, 2004.

\bibitem{BBI} Dimitri Burago, Yuri Burago and Sergei Ivanov.
\textit{A course in metric geometry.}
Volume 33 of
Graduate Studies in Mathematics. American Mathematical
Society, Providence, RI, 2001.

\bibitem{EG}Lawrence C. Evans and Ronald F. Gariepy.
\textit{Measure theory and fine properties of functions.}
Studies in Advanced Mathematics. CRC Press, Boca Raton, FL,
1992.

\bibitem{F}Bent Fuglede.
\textit{The Dirichlet problem for harmonic maps from Riemannian polyhedra to spaces
of upper bounded curvature.}
Trans. Amer. Math. Soc., 357(2):757-792, 2005.

\bibitem{GW} Changyu Guo and Stefan Wenger.
\textit{Area minimizing discs in locally non-compact metric spaces.}
Preprint \texttt{arXiv:1701.06736}, 2017.



\bibitem{HKST} Juha  Heinonen, Pekka Koskela, Nageswari Shanmugalingam  and Jeremy Tyson.
\textit{Sobolev spaces on metric measure spaces.}
Cambridge University Press,
2015.


\bibitem{J} J\"urgen Jost.
\textit{Equilibrium maps between metric spaces.}
Calc. Var. Partial Differential Equations, 2 (2):173-204, 1994. 




\bibitem{KS}Nicholas J. Korevaar and Richard M. Schoen.
\textit{Sobolev spaces and harmonic maps for metric space targets.}
Comm. Anal. Geom., 1(3-4):561-659, 1993.


\bibitem{LWint} Alexander Lytchak and Stefan Wenger.
\textit{Intrinsic structure of minimal discs in metric spaces.}
Geom. Topol., to appear. Preprint \texttt{arXiv:1602.06755}, 2016.

\bibitem{LWcurv} Alexander Lytchak and Stefan Wenger.
\textit{Isoperimetric characterization of upper curvature bounds.}
Preprint \texttt{arXiv:1611.05261}, 2016.

\bibitem{LWplateau} Alexander Lytchak and Stefan Wenger.
\textit{Area minimizing discs in metric spaces.}
Arch. Ration. Mech. Anal., 223(3):1123-1182, 2017.



\bibitem{Mese} Chikako Mese.
\textit{The curvature of minimal surfaces in singular spaces.}
Comm. Anal. Geom., 9:1-34, 2001.




\bibitem{Psemi} Anton Petrunin.
\textit{Semiconcave functions in  Alexandrov's geometry.}
Surveys in Differential Geometry, no. 92, 11, Int. Press, Somerville, MA, (2007), 137-201.

\bibitem {PS} Anton Petrunin and Stephan Stadler.
\textit{Metric minimizing surfaces revisited.}
Preprint \texttt{arXiv:1707.09635}, 2017.












\bibitem{R} Yurii G. Reshetnyak.
\textit{Two-dimensional manifolds of bounded curvature.}
In Geometry, IV, volume 70 of Encyclopaedia Math. Sci., pages 3-163,
245-250. Springer, Berlin, 1993.

\bibitem{Sturm} Karl-Theodor Sturm.
\textit{Ricci Tensor for Diffusion Operators and Curvature-Dimension Inequalities under Conformal Transformations and Time Changes.}
Preprint \texttt{arXiv:1401.0687}, 2014.

\end{thebibliography}
\end{document}